\documentclass{amsart}
\usepackage{enumerate}
\usepackage{graphicx}
\usepackage[]{color}

\newtheorem{theorem}{Theorem}[section]
\newtheorem{lemma}[theorem]{Lemma}
\newtheorem{proposition}[theorem]{Proposition}
\newtheorem{corollary}[theorem]{Corollary}
\newtheorem{problem}{Problem}[section]
\newtheorem{conjecture}{Conjecture}[]

\newtheorem{fact}{Fact}[section]
\newtheorem{note}{Note}[section]

\newtheorem{Lemma}{Lemma}[]

\theoremstyle{definition}
\newtheorem{definition}[theorem]{Definition}
\newtheorem{example}[theorem]{Example}

\theoremstyle{remark}
\newtheorem{remark}[theorem]{Remark}

\numberwithin{equation}{section}
\newif\ifskip
\skiptrue

\newcommand{\NN}{\mathbb{N}}

\newcommand{\cC}{\mathcal{C}}
\newcommand{\cD}{\mathcal{D}}
\newcommand{\aA}{\mathcal{A}}
\newcommand{\goesto}{\rightarrow}

\begin{document}

\title{Weakly Distinguishing Graph Polynomials on Addable Properties}

\author{Johann A. Makowsky}
\address{Department of Computer Science, Technion - IIT, Haifa, Israel}
\email{janos@cs.technion.ac.il}

\author{Vsevolod Rakita$^{\star}$}
\address{Department of Mathematics, Technion - IIT, Haifa, Israel}
\email{vsevolod@campus.technion.ac.il}
\thanks{($^{\star}$) Partially supported by a Jacobs Scholarship for excellent students.}



\date{October 12, 2019}


\begin{abstract}
A graph polynomial $P$ is weakly distinguishing if for almost all finite graphs 
$G$ there is a finite graph $H$ that is not isomorphic to $G$ with $P(G)=P(H)$. 
It is weakly distinguishing on a graph property $\mathcal{C}$
if for almost all finite graphs $G\in\mathcal{C}$ there is $H \in \cC$ that is not isomorphic to $G$ with $P(G)=P(H)$.
We give sufficient conditions on a graph property $\mathcal{C}$ for the 
characteristic, clique, independence, matching, and domination  and $\xi$ 
polynomials, as well as the Tutte polynomial and its specialisations, to be weakly distinguishing on $\mathcal{C}$. 
One such condition is to be addable and small in the sense of C. McDiarmid, A. Steger and D. Welsh (2005).
Another one is to be of genus at most $k$.
\end{abstract}

\maketitle

\tableofcontents
\newpage
\section{Introduction and Preliminaries}
We consider only, unless otherwise stated, simple (finite, loopless, undirected graphs with no parallel edges) 
graphs with vertices labeled $1,...,n$. 
A graph property is a family of graphs that is closed under isomorphisms. 
For a graph property $\mathcal{C}$, denote by $\mathcal{C}(n)$ the graphs of order $n$ in $\mathcal{C}$. 
We only consider properties such that $\mathcal{C}(n)$ is non empty for all sufficiently large $n$.

Let $P$ be a graph polynomial. We say that two non isomorphic graphs $G$ and $H$ are 
$P$-mates if $P(G)=P(H)$, and that $G$ is $P$ unique if it has no $P$ mates. 
$P$ is {\em trivial} if all graphs $G,H$ are $P$-mates.
$P$ is {\em complete} if all graphs $G$ are $P$-unique.

In this paper we investigate conditions which imply that almost all graphs in 
a graph property $\mathcal{C}$ have a $P$-mate. More formally, we give the following definitions:

Let $P$ be a graph polynomial, and denote by $\mathcal{G}(n)$ the family of graphs 
of order $n$ with vertices labeled $1,...n$, and by $U_P(n)$ the set of 
$P$ unique graphs of order $n$. 
\begin{definition}
$P$ is {\em weakly distinguishing} if 
$$\lim_{n\goesto \infty} \dfrac{|U_P(n)|}{|\mathcal{G}(n)|}=0$$
and that $P$ is {\em almost complete} if 
$$\lim_{n\goesto \infty} \dfrac{|U_P(n)|}{|\mathcal{G}(n)|}=1$$
\end{definition}

In \cite{bollobas2000contraction},
Bollob\'as, Pebody and Riordan conjectured:

\begin{conjecture}[BPR-conjecture]
\label{BPR}
The  chromatic and Tutte polynomials are almost complete. 
\end{conjecture}
In \cite{makowsky2017p}, the analogous question 
for $r$-regular hypergraphs 
was considered, and for $r \geq 3$ the conjecture was refuted.
In \cite{noy2003graphs} it was observed, as a remark in the conclusions, that
the independence polynomial $In(G;x)$, discussed in Section \ref{se:applications-independence}, is weakly distinguishing
on all finite graphs.
In \cite{rakita2019weakly}, it was proven that an infinite number of graph polynomials, 
among them the independence, clique and harmonious polynomials, are weakly distinguishing.

A natural way to approach the question whether a graph polynomial $P$ is weakly distinguishing, 
almost complete or otherwise, is to ask, given a graph property $\mathcal{C}$, 
whether almost all graphs in $\cC$ are $P$ unique in $\cC$.
Let $\mathcal{C}$ be a graph property and
denote by $U_{P,\mathcal{C}}(n)=\{G\in \mathcal{C}(n):G$ has no $P$-mate in $\cC \}$. 
\begin{definition}
$P$ is {\em weakly distinguishing on $\mathcal{C}$} if 
$$
\lim_{n\goesto \infty} \dfrac{|U_{P,\mathcal{C}}(n)|}{|\mathcal{C}(n)|}=0,
$$
and that $P$ is {\em almost complete on $\mathcal{C}$} if 
$$
\lim_{n\goesto \infty} \dfrac{|U_{P,\mathcal{C}}(n)|}{|\mathcal{C}(n)|}=1.
$$
\end{definition}

In this paper we prove that  many well studied graph polynomials
are weakly distinguishing in infinitely many graph properties $\cC$, Example \ref{lis}.
However, all these properties $\cC$ are small (in the sense that they are sets of measure $0$ 
in the collection of all graphs), 
so these results do not imply that any of the above polynomials are weakly distinguishing or 
almost complete for arbitrary properties $\cC$.

The graph polynomials for which we show that they are weakly distinguishing for $\cC$ 
include the characteristic polynomial, the  domination polynomial,
and the $\xi$-polynomial, which is a generalization of the both the matching and the Tutte polynomial.
These three polynomials are mutually incomparable in distinctive power by Proposition \ref{prop:incomparable}.
Once we know that $\xi$ is weakly distinguishing for a graph property $\cC$,
the holds also for all the graph polynomials which are substitution instances of $\xi$ in $\cC$.
This includes the matching polynomials, the independence polynomial, the chromatic polynomial, 
the Tutte polynomial, the Euler polynomial,
and many others, see Section \ref{se:applications}.


This paper is organized as follows: 
In Section \ref{se:dp} we introduce a method to compare distinctive power of graph polynomials and 
formulate how to use it (Lemma \ref{le:dp}.
In Section \ref{se:addable}, we give a graph theoretic background on addable properties, and graphs of genus at most $k$, 
and formulate our main tools, Lemma \ref{mcd2a} and Lemma \ref{mcd4a}.
In Sections \ref{se:char} and \ref{se:dom} we prove our results: 
Theorems \ref{thm:charMainTheorem},\ref{thm:charWeaklyOnGenus} for the characteristic polynomial,
Theorems \ref{thm:mainDomTheorem},\ref{thm:DomGenus} for the domination polynomial. 
In Section \ref{se:xi} we prove the same for the $\xi$-polynomial
and in Section we apply Lemma \ref{le:dp} to derive the corresponding results
for the matching polynomials, the independence polynomial, the chromatic polynomial, 
the Tutte polynomial, the Euler polynomial, and many others.


In Section \ref{se:conclu} we draw conclusions and present open problems.

\subsubsection*{Acknowledgements}
Preliminary versions of our results were presented at the 
19th International Conference on Random Structures and Algorithms,
15 to 19 July 2019, Z\"urich, Switzerland.
The results of Sections \ref{se:xi} and \ref{se:applications}
were partially obtained during Seminar 19401 in Dagstuhl (Comparative Theory for Graph Polynomials)
in October 2019. We want to thank the participants of the Dagstuhl Seminar
Delia Garijo,
Anna de Mier, Marc Noy, Elena Ravve and Peter Tittmann
for
valuable discussions and comments. Thanks go especially  to
Peter Tittmann for his guidance and comments about invariants derivable from $\xi$-polynomial, which helped us 
formulating Section \ref{se:applications}. 

\section{Comparing graph polynomials}
\label{se:dp}
In this section we provide a tool (Lemma \ref{obs:wd})
which allows us to show that many graph polynomials are weakly distinguishing.

\begin{definition}
Let $\cC$ be a graph property, $P$ and $Q$ be two graph polynomials and $G$ and $H$ two finite graphs.
\begin{enumerate}[(i)]
\item
$G$ and $H$ are {\em similar} if they have the same number of vertices, edges and connected components.
\item
$P <_{d.p}^{\cC} Q$ or
{\em  $Q$ is at least as distinctive as $P$ in $\cC$}
if
for all graphs $G,H \in \cC$, $Q(G;\bar{x})=Q(H;\bar{x})$ implies $P(G;\bar{y})=P(H;\bar{y})$.
\item
$P \sim_{d.p}^{\cC} Q$ or
{\em $P$ and $Q$ are of the same distinctive power in $\cC$} if
$P <_{d.p}^{\cC} Q$ and $Q <_{d.p}^{\cC} P$. 
\item
$P <_{s.d.p}^{\cC} Q$ or
{\em $P$ and $Q$ are of the same distinctive power in $\cC$ on similar graphs} if
if
for all similar graphs $G,H \in \cC$, $Q(G;\bar{x})=Q(H;\bar{x})$ implies $Q(G;\bar{y})=Q(H;\bar{y})$.
\item
$P \sim_{s.d.p}^{\cC} Q$  if
$P <_{s.d.p}^{\cC} Q$ and $Q <_{s.d.p}^{\cC} P$. 
\item
For all graph properties $\cC$
$P <_{d.p}^{\cC} Q$  implies $P <_{s.d.p}^{\cC} Q$.
\end{enumerate}
If $\cC$ consists of all finite graphs, we omit it.
\end{definition}

The partial preorders $<_{d.p.}$ and $<_{s.d.p.}$ between graph polynomials (or general graph invariants)
are studied extensively in
\cite{makowsky2019logician}.
A complete (trivial) graph polynomial is a maximal (minimal) element with respect to $<_{d.p.}$

\begin{example}
\label{ex:dp}
\begin{enumerate}[(i)]
\item
The chromatic polynomial $\chi(G;x)$ and the Tutte polynomial satisfy
$\chi(G;x) <_{s.d.p.} T(G;x,y)$ but not
$\chi(G;x) <_{d.p.} T(G;x,y)$, see Section \ref{se:applications}, because $T(;x,y)$ does not determine
the order of $G$ in the presence of isolated vertices.
\item
The characteristic polynomial $P_A(G;x)$ and the (defect aka acyclic) matching polynomial $\mu(G;x)$ 
from Section \ref{se:applications-matching}
are $d.p.$-equivalent on forests, see \cite{godsil1981theory}.
\item
Let $\bar{G}$ be the (loopless) complement graph of a simple graph $G$,
and $P(G; \bar{x})$ be a (possibly multivariate) graph polynomial. 
Put $\bar{P}(G; \bar{x}) = P(\bar{G}, \bar{x})$.
Then
$\bar{P}(G; \bar{x}) \sim_{d.p.} P(G, \bar{x})$.
\\
If we relativize this to a graph property $\cC$
$\bar{P}(G; \bar{x}) \sim_{d.p.}^{\cC} P(G, \bar{x})$ holds provided $\cC$ is closed under complement graphs.
\end{enumerate}
\end{example}

\begin{proposition}
\label{prop:incomparable}
The following graph polynomials are pairwise incomparable by $<_{s.d.p}$:
\begin{enumerate}[(i)]
\item 
The chromatic polynomial $\chi(G;x)$ and the independence polynomial $In(G;x)$ from Section \ref{se:applications}.
\item 
The characteristic polynomial $P_A(G,x)$ 
from Sections \ref{se:char}
and the chromatic polynomial $\chi(G;x)$.
\item 
The characteristic polynomial $P_A(G;x)$ and the independence polynomial $In(G;x)$.
\item 
The characteristic polynomial $P_A(G,x)$ and the domination polynomial $Dom(G;x)$.
from Sections \ref{se:char} and \ref{se:dom}.
\item 
The characteristic polynomial and the $\xi$-polynomial from Section \ref{se:xi}.
\item 
The domination polynomial and the $\xi$-polynomial.
\end{enumerate}
\end{proposition}
\begin{proof}
The first three examples are from \cite{makowsky2019logician}.

For (iv)-(vi) we first note that
all graphs of order less than $8$ 
and all trees of order less than $10$
are $\xi$-unique,
see \cite{trinks2011covered}.

(A):
Consider the graphs $C_4 \sqcup K_1$ and $S_5$ in Figure \ref{fig:cospectral3}.
We have\\
$P_A(S_5,x) = P_A(C_4 \sqcup K_1, x) = x^2(x+2)(x-2)$.
\\
$Dom(S_5, 1) =1$, $Dom(C_4 \sqcup K_1, 1)=0$,
and \\
$\xi(S_5, x) \neq \xi(C_4 \sqcup K_1, x)$, since they are of order $5$. 

\begin{figure}[h!]
        \caption{$P_A$-mates}
        \centering
        \includegraphics[scale=0.5]{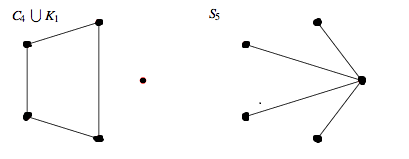}
        \label{fig:cospectral3}
\end{figure}

(B):
Let $P_5$ and $\hat{P_5}$ be the graphs of order $5$ in Figure \ref{fig:p5}.
used in Corollary \ref{cor:domMateP5}.
\\
$Dom(P_5,x) = Dom(\hat{P_5},x)$, since every dominating set of $\hat{P_5}$ is also a dominating set of $P_5$.
\\
$P_A(P_5,x)=-x(x-1)(x+1)(x^2-3) \neq P_A(\hat(P_5),x)=-x(x^2-x-3)(x^2+x-1)$,
hence
$P_A(P_5,x) \neq P_A(\hat{P_5},x)$.
\\
$\xi(P_5,x) \neq \xi(\hat{P_5},x)$, 
since both $P_5$ and $\hat{P_5}$ are $\xi$-unique.

\begin{figure}[h!]
        \caption{$\xi$-mates}
        \centering
        \includegraphics[scale=0.5]{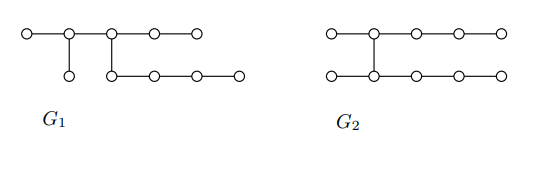}
        \label{fig:coXi}
\end{figure}
(C):
The graphs $G_1$ and $G_2$ Figure \ref{fig:coXi} are of order $10$.
Therefore
we have $\xi(G_1,x) = \xi(G_2,x)$.
\\
However,
$Dom(G_1;x)=x^{10}+10x^9+40x^8+82x^7+92x^6+56x^5+16x^4$
and
$Dom(G_2;x)=x^{10}+10x^9+41x^8+86x^7+94x^6+48x^5+9x^4$, hence
$Dom(G_1,x) \neq Dom(G_2,x)$ and
\\
$P_A(G_1;x)=x^2(x^4-x^3-4x^2+2x+3)(x^4+x^3-4x^2-2x+3)$
and
$P_A(G_2;x)=x^2(x-1)(x+1)(x^2-2)(x^4-5x^2+3)$, hence
$P_A(G_1,x) \neq P_A(G_2,x)$.

Now (A) and (B) proves (iv), (A) and (C) proves (v), and (B) and (C) proves (vi).
\end{proof}

In Section \ref{se:applications}
we use the following observation:
\begin{Lemma}
\label{obs:wd}
\label{le:dp}
Let $P(G; \bar{x})$
and $Q(G; \bar{y})$ two graph polynomials and $\cC$ and $\cD$ be graph properties with $\cD \subseteq \cC$.
Assume that 
$P <_{s.d.p}^{\cC} Q$ 
and $Q$ is weakly distinguishing in $\cC$. 
Then also $P$ 
is weakly distinguishing in $\cC$ but not necessarily in $\cD$.
\end{Lemma}

\begin{proof}
Clearly, $P <_{s.d.p}^{\cC} Q$ implies that
$U_P(n) \subseteq U_Q(n)$.
Hence
$$
\lim_{n\goesto \infty} \dfrac{|U_Q(n)|}{|\mathcal{G}(n)|}
\leq
\lim_{n\goesto \infty} \dfrac{|U_P(n)|}{|\mathcal{G}(n)|}.
$$
\end{proof}

\section{Addable Properties and graphs of genus $k$}
\label{se:addable}
We discuss properties on which the Tutte, Domination, Matching, $\xi$, Clique and Characteristic polynomials are 
weakly distinguishing. These properties were discussed in
\cite{mcdiarmid2005random} and \cite{mcdiarmid2008random}.

\subsection{Small addable classes}
\begin{definition}
We say a graph property $\mathcal{A}$ is decomposable if a graph is in $\mathcal{A}$ if and only if each of its connected components are in $\mathcal{A}$.\\
 We say a graph property $\mathcal{A}$ is bridge addable if for each graph $G\in \mathcal{A}$ and every two vertices $u,v$ in different components of $G$ the graph obtained by adding an edge between $u$ and $v$ is also in $\mathcal{A}$.\\
 We say a graph property $\mathcal{A}$ is addable if it is decomposable and bridge addable.
\end{definition}
\begin{example}
\label{lis}
The following properties are easily seen to be addable:
\begin{enumerate}[(i)]
\item Planar graphs;
\item Outerplanar graphs;
\item Series Parallel graphs;
\item Graphs with tree-width at most $k$ for $k \geq 2$;
\item $k$-colorable graphs for $k \geq 2$;
\item Graphs with no cycles of length greater than $k$;
\item Graphs with no $K_k$ minor for $k \geq 2$.
\end{enumerate}
\end{example}

Let $\aA$ be a graph property. If $\aA$ is minor closed (that is, closed under deletion of vertices and edges, and under contraction of edges) the Graph Minor theorem says that $\aA$ is characterized by a finite set of forbidden minors (see \cite{lovasz2006graph} for more on graph minors and the graph minor theorem). We can characterize addable  minor closed graph properties in terms of their forbidden minors:
\begin{proposition}[\cite{mcdiarmid2009random}]
Let $\mathcal{A}$ be a minor closed graph property. Then $\mathcal{A}$ is addable if and only if each excluded minor of $\mathcal{A}$ is 2-connected.
\end{proposition}
Note that any non empty, minor closed addable graph property $\aA$ contains all forests, as it contains the graph with a single vertex, and is closed under taking disjoint unions of graphs in $\aA$ and under adding edges between different components.

In addition to being addable, we will require the properties we consider to be small:
\begin{definition}
We say that a graph property $\aA$ is small if there exists a constant $a>0$ such that $|\aA_n|\leq a^nn!$ for all sufficiently large $n$.
\end{definition}
The following result of Norine et. al. is convenient:
\begin{theorem}[\cite{norine2006proper}]
\label{NSTW}
Let $\cC$ be a proper minor closed graph property. Then $\cC$ is small.
\end{theorem}

Let $H$ be a graph on vertex set $\{1,...,h\}$ and let $G$ be a graph on the vertex set $\{1,...,n\}$ where $n>h$. Let $W\subseteq V(G)$ with $|W|=h$, and let the root $r_W$ denote the least element in $W$. We say that $W$ is a pendant appearance of $H$ in $G$ if (a) the increasing bijection from $\{1,...,h\}$ to $W$ gives an isomorphism between $H$ and and the induced subgraph $G[W]$ of $G$, and (b) there exists exactly one edge between $W$ and $V(G)-W$, and this edge is incident with the root $r_W$.
Our method for proving graph polynomials are weakly distinguishing will relay on the following theorem:
\begin{theorem}[McDiarmid,Steger,Welsh \cite{mcdiarmid2005random}]
\label{mcd}
Let $\mathcal{C}$ be a non-empty, addable, and small graph property, let $H\in \mathcal{C}$ be a connected graph, and let $R_n$ be a random graph selected uniformly at random from the graphs of order $n$ in $\mathcal{C}$. Denote by $f_H(R_n)$ the number of pendant appearances of $H$ in $R_n$. Then there are constants $\alpha>0$,$n_0\in \mathbb{N}$ such that for all $n>n_0$,
$$
\mathbb{P}[f_H(R_n)\leq \alpha n]<e^{-\alpha n}
$$
\end{theorem}
For our purposes, we shall only need a weaker corollary of theorem \ref{mcd}:
\begin{corollary}
\label{mcd2}
Let $\mathcal{A}$ be an addable proper minor closed graph property, $H$ a fixed connected graph in $\mathcal{A}$. Denote by $\mathcal{A}'_n$ the set of all $n$ vertex graphs $G\in \mathcal{A}$ that have at least one pendent appearance of $H$. Then $\lim_{n \goesto \infty} \frac{|\mathcal{A}'_n|}{|\mathcal{A}_n|}=1$.
\end{corollary}
All the properties listed in example \ref{lis} are addable and minor closed, hence the corollary applies to them.
Using this corollary, we get:
\begin{Lemma}
\label{mcd2a}
Let $P$ be a graph polynomial, $\aA$ a small addable graph property. If there is a fixed $H\in \aA$ such that if a graph $G\in \aA$ has a pendant appearance of $H$, $G$ has a $P$ mate $G'\in \aA$, then $P$ is weakly distinguishing on $\aA$.
\end{Lemma}
\begin{proof}
By corollary \ref{mcd2}, almost all graphs $G\in \aA$ have a pendant appearance of $H$, and hence almost all graphs $G\in \aA$ have a $P$-mate, so $P$ is weakly distinguishing on $\aA$.
\end{proof}

\subsection{Graphs With Genus at Most $k$}
Let $G$ be a graph.The genus of $G$, denoted $g(G)$, is the minimal genus of a surface in which $G$ can be embedded (for more on graphs embedded in surfaces see, for example, \cite{lando2013graphs,mohar2001graphs}). It is easy to see that $g(G)$ is well defined, as for all $G$, $g(G)\leq |E(G)|$. For $k\in \NN$, denote by $\cC_k$ the property of graphs with genus at most $k$. Note that in general $\cC_k$ is not addable - for instance, the genus of $K_5$, the complete graph with 5 vertices, is 1, but the genus of a disjoint union of two copies of $K_5$ has genus 2- so we can not apply theorem \ref{mcd} to $\cC_k$. However, the same result does hold for $\cC_k$ with a slight modification:

\begin{theorem}[McDiarmid \cite{mcdiarmid2008random}]
\label{mcd3}
Let $H$ be a connected graph, $k\in \NN$, and let $R_n$ be a random graph selected uniformly at random from the graphs of order $n$ in $\mathcal{C}_k$. Denote by $f_H(R_n)$ the number of pendant appearances of $H$ in $R_n$. Then there are constants $\alpha>0$,$n_0\in \mathbb{N}$ such that for all $n>n_0$,
$$
\mathbb{P}[f_H(R_n)\leq \alpha n]<e^{-\alpha n}
$$
\end{theorem}
Again, we only need a weaker corollary of this theorem:
\begin{corollary}
\label{mcd4}
Let $H$ a fixed connected planar graph, $k\in \NN$ and let $\aA=\cC_k$. Denote by $\mathcal{A}'_n$ the set of all $n$ vertex graphs $G\in \mathcal{A}$ that have at least one pendent appearance of $H$. Then $\lim_{n \goesto \infty} \frac{|\mathcal{A}'_n|}{|\mathcal{A}_n|}=1$.
\end{corollary}
Using this corollary, we get:
\begin{Lemma}
\label{mcd4a}
Let $P$ be a graph polynomial, $\mathcal{C}_k$ the property of graphs with genus at most $k$. If there is a planar graph $H$ such that if a graph $G$ has a pendant appearance of $H$, $G$ has a $P$ mate $G' \in \cC_k$, then $P$ is weakly distinguishing on $\cC_k$.
\end{Lemma}
\begin{proof}
By corollary \ref{mcd4}, almost all graphs $G\in \cC_k$ have a pendant appearance of $H$, and hence almost all graphs $G\in \cC_k$ have a $P$-mate, so $P$ is weakly distinguishing on $\cC_k$.
\end{proof}

\section{The Characteristic Polynomial}
\label{se:char}
\label{sec:MainCharacteristic}
\begin{definition}
Let $G$ be a graph, and denote by $A_G$ its adjacency matrix. 
The characteristic polynomial of $G$, denoted $P_A(G)$ is defined as the characteristic polynomial of $A_G$. 
The roots of $P_A(G)$ are referred to as the spectrum of $G$.\\
A $P_A$-unique graph is usually referred to as a graph determined by its spectrum, 
and if $G$ and $H$ are two non isomorphic graphs such that $P_A(G)=P_A(H)$, $G$ and $H$ are said to be cospectral.
\end{definition}
The characteristic polynomial and particularly its roots are a widely studied topic, see e.g. 
\cite{brouwer2011spectra} for an introduction.
A classic result of Schwenk states that almost all trees are not determined by their spectrum 
(see \cite{schwenk1973almost} for details). 
Using corollary \ref{mcd2}, we can extend this result to all small addable graph properties.

We use the following recurrence relation for the characteristic polynomial:
\begin{lemma}[see \cite{clarke1970characteristic}]
\label{prop:charGenRec}
Let $G_1$ and $G_2$ be two graphs, and let $v_1\in V(G_1)$ ,$v_2 \in V(G_2)$. Denote by $H$ the graph obtained from the disjoint union of $G_1$ and $G_2$ by adding an edge from $v_1$ to $v_2$. Then
\begin{align}
\label{eq:char1}
P_A(H,x)=P_A(G_1,x)P_A(G_2,x)-P_A(G_1-v_1,x)P_A(G_2-v_2,x)
\end{align}
\end{lemma}

Consider the graphs in Figure \ref{fig:cospectral} (considered by Schwenk in \cite{schwenk1973almost}).
\begin{figure}[h!]
	\caption{Cospectral Graphs}
	\centering
	\includegraphics[scale=0.5]{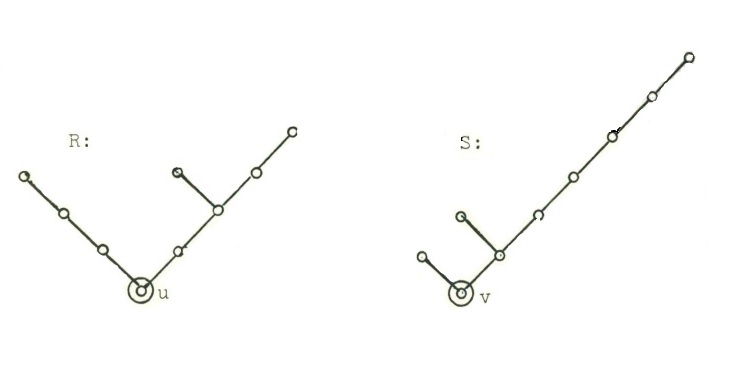}
	\label{fig:cospectral}
\end{figure}

Note that $S$ and $R$ are isomorphic, and we can check  that 
$$P_A(S-v,x)=P_A(R-u,x)=x^8-6x^6+10x^4-4x^2$$ 
Thus, we get:
\begin{lemma}
\label{lemma:charMate}
Let $G_S$ be a graph with a pendent appearance of $S$ rooted at $v$. Denote by $G_R$ the graph obtained from $G$ by replacing the pendant appearance of $S$ in $G$ with a pendant appearance of $R$, rooted at $u$. Then $G_S$ and $G_R$ are cospectral mates.
\end{lemma}
\begin{proof}
Denote by $w\in V(G_S)$ the vertex adjacent to $v$ that is not in the pendant appearance of $S$, and by $H$ the graph obtained from $G_S$ by deleting all the vertices in the pendant appearance of $S$. By applying relation \ref{eq:char1} to $G_S$  with the edge $wv$, we get
$$
P_A(G_S,x)=P_A(S,x)P_A(H,x)-P_A(S-v,x)P_A(H-w,x)
$$
By applying relation \ref{eq:char1} to $G_R$ with the edge $wu$ we get
$$
P_A(G_R,x)=P_A(R,x)P_A(H,x)-P_A(R-u,x)P_A(H-w,x)
$$
But since $P_A(R,x)=P_A(S,x)$ and $P_A(S-v,x)=P_A(R-u,x)$, we get that $P_A(G_S,x)=P_A(G_R,x)$, so $G_S$ and $G_R$ are cospectral mates.
\end{proof}
Note that $S$ and $R$, as unrooted graphs, are isomorphic. Thus, we have:
\begin{theorem}
\label{thm:charMainTheorem}
Let $\mathcal{A}$ be a small addable graph property such that $S\in \aA$. Then $P_A$ is weakly distinguishing on $\mathcal{A}$.
\end{theorem}
\begin{proof}
From lemmas \ref{mcd2a} and \ref{lemma:charMate} we get that $P_A$ is weakly distinguishing on $\aA$.
\end{proof}
\begin{corollary}
Let $\mathcal{A}$ be a proper minor closed addable graph property. Then $P_A$ is weakly distinguishing on $\mathcal{A}$.
\end{corollary}
\begin{proof}
Since all forests are in $\mathcal{A}$, and $S$ is a tree, by theorems \ref{thm:charMainTheorem} and \ref{NSTW}, $P_A$ is weakly distinguishing on $\mathcal{A}$.
\end{proof}
Since all the properties in list \ref{lis} are addable and minor closed, we have:
\begin{corollary}
$P_A$ is weakly distinguishing on all the properties listed in example \ref{lis}.
\end{corollary}
Similarly, we get:
\begin{theorem}
\label{thm:charWeaklyOnGenus}
Denote by $\cC_k$ the class of graphs of genus at most $k$. Then $P_A$ is weakly distinguishing on $\cC_k$ for all $k\in \NN$.
\end{theorem}
\begin{proof}
Since $S$ is a tree, from lemmas \ref{mcd4a} and \ref{lemma:charMate} we get that $P_A$ is weakly distinguishing on $\cC_k$.
\end{proof}

Instead of the characteristic polynomial $P_A$, which is based on the adjacency matrix,
one can also look at the analogue graph polynomial $P_L$ which is based on the Laplacian matrix 
$L_G = D_G - A_G$, where $D_G$ is the diagonal matrix with diagonal elements $d_{v,v}= deg(v)$ for
each $v \in V(G)$. Now $P_L(G;x)$ is the characteristic polynomial of $L_G$.
An early survey about the Laplacian polynomial may be found in \cite{mohar1991laplacian}.
The graph polynomials are $d.p.$-equivalent on regular graphs, but $d.p.$-incomparable on all finite graphs,
cf. \cite{brouwer2011spectra} and \cite{makowsky2019logician}.
Our proofs of Theorems \ref{thm:charMainTheorem} and \ref{thm:charWeaklyOnGenus} do not work for $P_L(G;x)$.

\begin{problem}
Find two graphs $R$ and $S$ which can be used two prove the analogue of 
Theorems \ref{thm:charMainTheorem} and \ref{thm:charWeaklyOnGenus}.
\end{problem}

\section{The Domination Polynomial}
\label{se:dom}
\label{sec:MainDomination}
The Domination Polynomial is the generating function of dominating sets in a given graph $G$. More formally,
\begin{definition}
Let $G$ be a graph. A set $S \subseteq V(G)$ is called a dominating set if for every $v\in V(G)$, either $v\in S$ or $v$ has a neighbor $u\in S$. Define $Dom(G;x)=\sum_{S\subseteq V(G)} x^{|S|}$ where the sum is over all dominating sets of $G$.
\end{definition}

The Domination Polynomial was extensively studied in recent years, see \cite{alikhani2009introduction} for a survey. In \cite{kotek2012recurrence}, the following was observed:
\begin{lemma}
\label{leaves}
Let $G$ be a graph. We say a vertex $v\in V(G)$ is a stem if it has a neighbor $u\in V(G)$ with $deg(u)=1$. Assume $G$ has two stems $u,u'\in G$, $(u,u')\in E(G)$, and denote by $G'$ the graph resulting from $G$ by deleting the edge $(u,u')$. Then for every set $S\subseteq V(G)$, $S$ is a dominating set of $G$ if and only if $S$ is a dominating set of $G'$.
\end{lemma}
\begin{proof}
If $S$ is a dominating set of $G'$, it is clearly also a dominating set of $G$. On the other hand, let $S$ be a dominating set of $G$. After deleting $(u,u')$, the only vertices that are perhaps not dominated by $S$ are $u,u'$. But $u$ has a vertex of degree 1 as a neighbor, so either it or $u$ are in $S$, and in either case $u$ is dominated. The same applies to $u'$. So $S$ is a dominating set of $G'$.
\end{proof}
From this lemma, we get:
\begin{theorem}
Every graph $G$ that has two distinct stems has a $Dom$-mate.
\end{theorem}
\begin{proof}
This is a direct consequence of the lemma.
\end{proof}
Thus, we have:
\begin{corollary}
\label{cor:domMateP5}
Let $G$ be a graph that has a pendant appearance of $P_5$ rooted at $r$ (see Figure \ref{fig:p5}), and let $\hat{G}$ be the graph obtained from $G$ by replacing the pendant appearance of $P_5$ with a pendant appearance of $\hat{P}_5$ rooted at $r$. Then $G$ and $\hat{G}$ are $Dom$-mates.
\end{corollary}

\begin{figure}[h!]
	\caption{A path of length 5}
	\centering
	\includegraphics[scale=1]{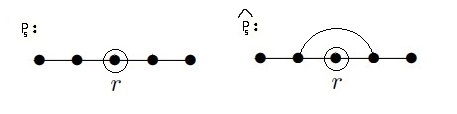}
\label{fig:p5}
\end{figure}

And so:
\begin{theorem}
\label{thm:mainDomTheorem}
Let $\mathcal{A}$ be a small addable graph property such that $P_5,\hat{P}_5\in \mathcal{A}$. Then $Dom$ is weakly distinguishing on $\mathcal{A}$.
\end{theorem}
\begin{proof}
From lemma \ref{mcd2a} and corollary \ref{cor:domMateP5} we get that $Dom$ is weakly distinguishing on $\aA$.
\end{proof}
\begin{corollary}
Let $\mathcal{A}$ be a proper minor closed addable graph property, such that $\hat{P}_5\in \mathcal{A}$. Then $Dom$ is weakly distinguishing on $\mathcal{A}$
\end{corollary}
\begin{proof}
Note that, since all forests are in $\aA$, $P_5\in \aA$, and hence from theorems \ref{thm:mainDomTheorem} and \ref{NSTW} $Dom$ is weakly distinguishing on $\mathcal{A}$. 
\end{proof}
As in the previous sections, we have:
\begin{corollary}
$Dom$ is weakly distinguishing on all the properties listed in example \ref{lis}.
\end{corollary}
Similarly, we get:
\begin{theorem}
\label{thm:DomGenus}
Denote by $\cC_k$ the property of graphs of genus less than $k$. Then $Dom$ is weakly distinguishing on $\cC_k$ for all $k\in \NN$.
\end{theorem}
\begin{proof}
Since $P_5$ and $\hat{P}_5$ are planar, from lemma \ref{mcd4a} and corollary \ref{cor:domMateP5} we get that $Dom$ is weakly distinguishing on $\cC_k$.
\end{proof}

\section{The $\xi$ Polynomial}
\label{se:xi}
The $\xi$ polynomial, introduced in \cite{averbouch2008most}\cite{averbouch2008extension}, 
generalizes the matching and the Tutte polynomial. It is defined via a recurrence relation:
\begin{definition}
Let $G$ be a graph. Define the tri-variate polynomial $\xi(G;x,y,z)$ using the recursive relation
\begin{align}
\label{eq:xirec}
\xi(G;x,y,z)=\xi(G-e;x,y,z)+y\xi(G/e;x,y,z)+z\xi(G\dagger e;x,y,z)
\end{align}
$$
\xi(G\sqcup H;x,y,z)=\xi(G;x,y,z)\xi(H;x,y,z)
$$
with the base conditions $\xi(K_1;x,y,z)=x$ and $\xi(\emptyset;x,y,z)=1$.
\end{definition}

Let $G =(V(G),E(G))$ be a graph and $A \subseteq E(G)$. We denote by $G\langle A \rangle$ the graph
$(V(G), A)$, the spanning subgraph of $A$ in $G$.
An alternative representation of the $\xi$ polynomial was introduced by Trinks in \cite{trinks2011covered}:
\begin{definition}
Let $G$ be a graph. The covered components polynomial $C(G;x,y,z)$ of $G$ is defined as:
$$
C(G;x,y,z)=\sum_{A\subseteq E(G)}x^{k(G\langle A \rangle )}y^{|A|}z^{c(G\langle A \rangle )}
$$
were $G\langle A \rangle $ is the graph $(V(G),A)$, 
$k(G\langle A \rangle )$ is the number of connected components of $G\langle A \rangle $ and $c(G\langle A \rangle )$ 
is the number of covered connected components of $G\langle A \rangle $, that is connected components in $G\langle A \rangle $ with at least one edge.
\end{definition}

We sometimes write $C(G)$ for $C(G;x,y,z)$ to simplify long expressions.
The two definitions are connected via the following result:

\begin{proposition}[\cite{trinks2011covered}]
For all graphs $G$,
\begin{align}
C(G;x,y,z)=\xi(G;x,y,xyz-xy)
\end{align}
\end{proposition}

\begin{corollary}
\label{cor:xiDistC}
Let $G$ and $H$ be graphs. Then $C(G;x,y,z)=C(H;x,y,z)$ if and only if $\xi(G;x,y,z)=\xi(H;x,y,z)$
\end{corollary}
We will use the following recurrence relation:

\begin{theorem}[\cite{trinks2011covered}]
\label{thm:Crec}
Let $G_1,G_2$ be graphs, and $v_1\in V(G_1),v_2\in V(G_2)$. 
Let $H$ be the graph obtained from $G_1$ and $G_2$ by identifying $v_1$ with $v_2$. Then the following holds:
\begin{gather}
\label{eq:Crec}
C(G)=(\frac{1}{xz}+\frac{2}{x})C(G_1)C(G_2)+
\notag\\
(-\frac{1}{z}-1)\left[ C(G_1)C(G_2-v_2)+C(G_1-v_1)C(G_2)\right]+ 
\notag\\
(\frac{x}{z}+x)C(G_1-v_1)C(G_2-v_2)
\end{gather}
\end{theorem}

Consider the graph $S$ in Figure \ref{pseudosimilar}.
\begin{figure}[h!]
        \caption{A tree with pseudosimilar vertices}
        \centering
        \includegraphics[scale=0.75]{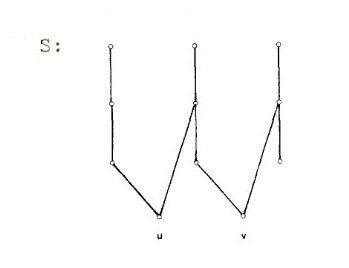}
        \label{pseudosimilar}
\end{figure}

We use the graphs in Figure \ref{pseudosimilar}, together with theorem \ref{thm:Crec} to prove:
\begin{lemma}
\label{lemma:CMates}
Let $G_v$ be a graph with a pendent appearance of $S$ (depicted in Figure \ref{pseudosimilar}) 
rooted at $v$, and let $G_u$ be the graph obtained from $G_v$ by replacing 
the pendant appearance of $S$ rooted at $v$ by a pendant appearance of $S$ rooted at $u$. 
Then $G_v$ and $G_u$ are $C$ mates.
\end{lemma}
\begin{proof}
We apply relation \ref{eq:Crec} to $G_v$, with $G_1=S$ and $G_2$ being the graph $H_1$ 
obtained from $G_v$ by replacing the pendant appearance of $S$ with a single vertex $w$. We get:
\begin{gather}
C(G_v)=(\frac{1}{xz}+\frac{2}{x})C(S)C(H_1)+
\notag\\
(-\frac{1}{z}-1)\left[ C(S)C(H_1-w)+C(S-v)C(H_1)\right]+
\notag\\
(\frac{x}{z}+x)C(S-v)C(H_1-w)
\end{gather}
Similarly, we apply relation \ref{eq:Crec} to $G_u$, 
with $G_1=S$ and $G_2$ being the graph $H_2$ obtained from $G_u$ by 
replacing the pendant appearance of $S$ with a single vertex $w$. We get:
\begin{gather}
C(G_u)=(\frac{1}{xz}+\frac{2}{x})C(S)C(H_2)+
\notag\\
(-\frac{1}{z}-1)\left[ C(S)C(H_2-w)+C(S-u)C(H)\right]+
\notag\\
(\frac{x}{z}+x)C(S-u)C(H_2-w)
\notag
\end{gather}
Note that $H_1 \cong H_2$, $H_1-w \cong H_2-w$, and $S-v \cong S-u$, and hence $C(G_v)=C(G_u)$.
\end{proof}
As a direct consequence of lemma \ref{lemma:CMates} and corollary \ref{cor:xiDistC}, we have:
\begin{corollary}
\label{cor:xiMates}
Let $G_v$ be a graph with a pendent appearance of $S$ (depicted in Figure \ref{pseudosimilar}) rooted at $v$, and let $G_u$ be the graph obtained from $G_v$ by replacing the pendant appearance of $S$ rooted at $v$ by a pendant appearance of $S$ rooted at $u$. Then $G_v$ and $G_u$ are $\xi$ mates.
\end{corollary}

Thus, we have:
\begin{theorem}
\label{thm:mainXiTheorem}
Let $\mathcal{A}$ be a small addable graph property such that $S\in \aA$. Then $\xi$ is weakly distinguishing on $\mathcal{A}$. 
\end{theorem}
\begin{proof}
From lemma \ref{mcd2a} and corollary \ref{cor:xiMates} we get that $\xi$ is weakly distinguishing on $\aA$.
\end{proof}
\begin{corollary}
Let $\mathcal{A}$ be a proper minor closed addable graph property. Then $\xi$ is weakly distinguishing on $\mathcal{A}$.
\end{corollary}
\begin{proof}
Note that, since all forests are in $\aA$, $S\in \aA$, and hence from theorems \ref{thm:mainXiTheorem} and \ref{NSTW}, $M$ is weakly distinguishing on $\mathcal{A}$.
\end{proof}
As in the previous section, we get
\begin{corollary}
$\xi$ is weakly distinguishing on all the properties listed in example \ref{lis}.
\end{corollary}
Similarly, we get:
\begin{theorem}
\label{thm:XiGenus}
Denote by $\cC_k$ the class of graphs of genus less than $k$. Then $\xi$ is weakly distinguishing on $\cC_k$ for all $k\in \NN$.
\end{theorem}
\begin{proof}
Since $S$ is a tree, from lemmas \ref{mcd4a} and \ref{cor:xiMates} we get that $\xi$ is weakly distinguishing on $\cC_k$.
\end{proof}

\section{$\xi$-invariants}
\label{se:applications}

In this section we investigate consequences of Theorems \ref{thm:mainXiTheorem} and \ref{thm:XiGenus} 
by using Lemma \ref{le:dp}.


From Theorem
\ref{thm:mainXiTheorem}
we get:

\begin{corollary}
Let $\mathcal{A}$ be a small addable graph property such that $S\in \aA$, 
and let $P$ be a graph polynomial such that 
$P <_{s.d.p}^{\cC} \xi$.
Then $P$ is weakly distinguishing on $\mathcal{A}$.
\end{corollary}

\begin{corollary}
Let $\mathcal{A}$ be a proper minor closed addable graph property, 
and let $P$ be a graph polynomial such that 
$P <_{s.d.p}^{\cC} \xi$.
Then $P$ is weakly distinguishing on $\mathcal{A}$.
\end{corollary}

From Theorem
\ref{thm:XiGenus}
we get:

\begin{corollary}
\label{cor:XiGenus}
Let $\cC_k$ the class of graphs of genus less than $k$, 
and let $P$ be a graph polynomial such that
$P <_{s.d.p}^{\cC} \xi$.
Then $P$ is weakly distinguishing on $\cC_k$ for all $k\in \NN$.
\end{corollary}

In the sequel of this section we collect results from the literature which show that
there are many graph polynomials $P$ where we can apply this.
Among them we find the
\begin{enumerate}[(i)]
\item
generalized chromatic polynomial from \cite{Dohmen-etal},
including the chromatic polynomial;
\item
the matching polynomials; 
\item
independence polynomial,
including the vertex cover polynomial;
\item
Tutte polynomial,
including the flow polynomial, the reliability polynomial and the Euler polynomial.
\end{enumerate}

\subsection{The generalized chromatic polynomial}
The generalized chromatic polynomial $GC(G;x,y)$ was introduced in \cite{Dohmen-etal}.
There are two disjoint sets of colors $Y$, the set of proper colors, and $Z$. 
A generalized coloring of a graph $G = (V , E )$ is a map $c : V \rightarrow (Y \sqcup Z )$ 
such that for all
$(u, v) \in E$, if $c(u) \in Y$ and $c(v) \in Y$ , then $c(u) = c(v)$. 
For two positive integers $x > y$, the value of the polynomial $GC(G;x,y)$ is the number of generalized colorings
by $x$ colors, where $y$ of them are proper. 
The chromatic polynomial $\chi(G;x)$ is obtained for the case $x=y$.

\begin{theorem}[{\cite[Proposition 22]{averbouch2008extension}}]
For all graphs $G$
$$GC(G;x,y) = \xi(G; x-1, x-y).$$
Therefore $\chi <_{d.p}^{\cC} CG <_{d.p}^{\cC} \xi$.
\end{theorem}

\subsection{The matching polynomial}
\label{se:applications-matching}
There are several variants of the matching polynomial considered in the literature:
\begin{definition}
\label{def:matching}
Let $G$ be a graph with $n$ vertices. 
A {\em matching} in $G$ is a spanning subgraph of $G$ in which every connected component 
is either an isolated vertex or two vertices connected by a single edge. 
We say a matching is of size $k$ if it has exactly $k$ edges. 
Denote by $m_k(G)$ the number of $k$ matchings in $G$.\\

The {\em matching acyclic polynomial} (also known as the matching defect polynomial) 
is defined as $$\mu(G;x)=\sum_{k=0}^{\lfloor n/2 \rfloor}(-1)^km_k(G)x^{n-2k}$$
The {\em matching generating polynomial} is defined as $$g(G;x)=\sum_{k=0}^{\lfloor n/2 \rfloor} m_k(G)x^k$$
The {\em bivariate matching polynomial} of $G$ is defined as 
$$M(G;w_1,w_2)=\sum_{k=0}^{\lfloor n/2 \rfloor}m_k(G)w_1^{n-2k}w_2^k$$
\end{definition}
For an introduction to the bivariate matching polynomial, 
see \cite{farrell1979introduction}. 
For a recent survey on the acyclic matching polynomial, see \cite{gutman2016survey}.\\
For our purposes, we note the following fact:
\begin{fact}
\label{fact:matching pols}
All of the polynomials in definition \ref{def:matching} are of the same distinctive power on similar graphs:. 
$$
\mu \sim_{s.d.p} g \sim_{s.d.p} M.
$$
In fact, we also have $\mu \sim_{d.p}  M$, but for the edgless graphs $E_n$ of order $n$ we have 
$g(E_n,x) = g(E_m)$ for all $m,n \geq 1$, but $\mu(E_n) \neq \mu(E_m)$.
\end{fact}
Thus we will only consider the bivariate matching polynomial.

\begin{theorem}[{\cite[Proposition 20]{averbouch2008extension}}]
For all graphs $G$
$$M(G;x,y) = \xi(G; x, 0, y),$$
therefore $M <_{s.d.p} \xi$.
\end{theorem}

\subsection{Independence and clique polynomial}
\label{se:applications-independence}
The independence polynomial is defined as
$$
In(G;x) = \sum_{A \subseteq E(G)}  x^{|A|}
$$
where the graph induced graph $G[A]$ is edgless.

The vertex cover polynomial $V(G;x)$ is defined as
$$
VC(G;x) = \sum_{A \subseteq E(G)}  x^{|A|}
$$
where $A$ is a vertex cover of $G$.

The following is taken from (but possibly not originally due to) \cite{trinks2011covered}.

\begin{proposition}
For all graphs $G$ we have
\begin{enumerate}[(i)]
\item
$In(G;x) = GC(G; x+1, 1)$
\item
$VC(G;x) = x^n In((G; \frac{1}{x})$
\end{enumerate}
Hence, $VC \sim_{s.d.p} In <_{s.d.p} GC$.
\end{proposition}

The clique polynomial is defined as
$$
Cl(G;x) = \sum_{A \subseteq E(G)}  x^{|A|}
$$
where the graph induced graph $G[A]$ is a complete graph.

We note that 
$In(G;x) = Cl(\bar{G})$ for simple graphs.
Therefore $In \sim_{d.p} Cl$ by Example \ref{ex:dp}(iii).

Both, $In(G;x)$ and $Cl(G;x)$ were shown in \cite{rakita2019weakly} to be weakly distinguishing on all finite graphs.
In the light of the above, $In(G;x)$ is also weakly distinguishing on small addable graph properties,
and on graphs of genus at most $k$.

This also holds
if $\cC$ is addable, small, and closed under complements.
If $\cC$ is addable and small, but not closed under complements 
we can use the following lemma:
\begin{figure}[h!]
	\caption{A path of length 3}
	\centering
	\includegraphics[scale=0.75]{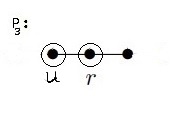}
	\label{p3}
\end{figure}
\begin{lemma}
\label{lemma:ClOnAddable}
Let $G_r$ be a graph such that $G$ has a pendant appearance of $P_3$ (see Figure \ref{p3}) rooted at $r$, and $G_u$ the graph obtained from $G_r$ by replacing the pendant appearance of $P_3$ by a pendant appearance of $P_3$ rooted at $u$. Then $G_r$ and $G_u$ are clique mates.
\end{lemma}
\begin{proof}
For all $k>2$, $G_r$ and $G_u$ have the same number of $k$ cliques (since no $k$ clique can include vertices from the pendant appearance), and $G$ and $G'$ have the same number of vertices and edges, and hence the same number of $1$ and $2$ cliques. So $Cl(G_r;x)=Cl(G_u;x)$.
\end{proof}

Thus, we have:
\begin{theorem}
\label{thm:mainCLTheorem}
Let $\mathcal{A}$ be a small addable graph property such that $P_3\in \mathcal{A}$. Then $Cl$ is weakly distinguishing on $\mathcal{A}$.
\end{theorem}
\begin{proof}
From Lemmas \ref{mcd2a} and \ref{lemma:ClOnAddable} we get that $Cl$ is weakly distinguishing on $\aA$.
\end{proof}

\subsection{The Tutte polynomial}
Let $G=(V(G), E(G)$ be a graph and $A \subseteq E(G)$.
We denote by $G\langle A \rangle $ the graph $(V(G), A)$.
$k(A)$ is the number of connected components of $G\langle A \rangle $.
The Tutte polynomial $T(G; x,y)$ is defined by
$$
T(G; x,y) = \sum_{A \subseteq E(G)}  (x-1)^{k(A)- k(E)} (y-1)^{k(A)+|A|-|V(G)|}.
$$
The partition function $Z(G; q,w)$ is defined by
$$
Z(G; q,w) = \sum_{A \subseteq E(G)} q^{k(F)} w^{|F|}.
$$ 
They are related by the equation
$$
T(G; x,y) = (x-1)^{-k(E)} (y-1)^{-|V(G)|} Z(G; (x-1)(y-1), (y-1)).
$$
The chromatic polynomial $\chi(G;x)$ can be obtained from the Tutte polynomial by
$$
\chi(G;x) = (-1)^{|V(G)| -k(G)} x^{k(G)} T(G; (1-x) , 0)
$$
From this we get

\begin{proposition}
$\chi <_{s.d.p} T \sim_{s.d.p} Z <_{d.p} C \sim_{d.p.}  \xi$.
\end{proposition}

A graph is Eulerian if all its vertices have even degree. It does not have to be connected.
The Euler polynomial $\mathcal{E}(G;x)$ of a graph is defined by
$$
\mathcal{E}(G;x) = \sum_{A \subseteq E(G): (V,A)  \mbox{ is Eulerian }} x^{|A|}.
$$
From \cite[Chapter 10, p.468]{aigner2007course} we get
that it is related to the Tutte polynomial by
$$
\mathcal{E}(G;x) =  (1-x)^{|E(G)-|V(G)|+k(G)} x^{|V(G)|-k(G)} T(G; \frac{1}{x}, \frac{1+x}{1-x}).
$$
Hence we get:
\begin{proposition}
$\mathcal{E} <_{s.d.p} T$.
\end{proposition}

The flow polynomial $Fl(G,x)$ and the reliability polynomial $R(G;p)$ are related to the Tutte
polynomial by
\begin{gather}
FL(G;x) = (-1)^{|E(G)-|V(G)|+k(G)} T(G; 0, 1-x) \notag \\
R(G;p) = (p)^{|E(G)-|V(G)|+k(G)} (1-p)^{|V(G)|-k(G)}  T(G; 1, \frac{1}{p}) \notag 
\end{gather}

Hence we get:
\begin{proposition}
$Fl <_{s.d.p} T$  and $R <_{s.d.p} T$ .
\end{proposition}

\section{Conclusion}
\label{se:conclu}
We have shown that many graph polynomials are weakly distinguishing on all 
proper minor closed addable graph properties, including many interesting properties such as 
planar graphs, graphs with tree-width at most $k$, and $K_k$ free graphs (for $4<k\in \NN$). 
In addition, we proved that
the Domination, Characteristic, and the edge elimination polynomial $\xi(G;x,y,z)$
are weakly distinguishing on the properties of graphs 
with genus less then $k$ for all $k\in \NN$.
This also applies to graph polynomials derivable from $\xi$, such as the generalized chromatic polynomial,
Tutte polynomial and its variants, the matching, the independence, and the clique polynomials.

Our results relied on the fact that proper minor closed addable properties are, 
in a sense, small, and so they do not settle the question whether the Domination, Characteristic or Tutte polynomials 
are weakly distinguishing, almost complete, or otherwise on all graphs.

We have shown that for the above graph polynomials $P$, 
the sequence $\alpha_P^\cC(n)=\frac{|U_P^\cC(n)|}{|\cC(n)|}$ tends to $0$ as $n$ tends to infinity.

The following questions are natural extensions of the work in this paper:
\begin{problem}
What can be said about $\alpha_P^\cC(n)$ when $\cC$ is assumed to be a hereditary or monotone graph property?
\end{problem}
\begin{problem}
Can we find a graph polynomial $P$ and a graph property $\cC$ such that we can prove 
that $\alpha_P^\cC(n)\geq \beta$ for all sufficiently large $n$ for some fixed $\beta\in (0,1]$?
\end{problem}
\begin{problem}
For $P$ one of the above polynomials and $\aA$ a proper minor closed addable property, 
select a random graph $G_n$ uniformly at random in $\aA_n$, and denote $[G_n]=\{H\in\aA_n:P(G)=P(H)\}$. 
What can be said about the limit distribution of the random variable $|[G_n]|$?
\end{problem}

\bibliographystyle{plain}
\bibliography{RR-references}

\end{document}